\documentclass[11pt,a4paper]{article}
\usepackage[affil-it]{authblk}
\usepackage[utf8]{inputenc}
\usepackage{geometry}
\usepackage[titletoc,toc,title]{appendix}
\usepackage{graphicx}
\usepackage{amssymb}
\usepackage{epstopdf}
\usepackage{mathrsfs}
\usepackage{amsmath}
\usepackage{amsthm}
\begin{document}
\newtheorem{theoreme}{Theorem}
\newtheorem{ex}{Example}
\newtheorem{definition}{Definition}
\newtheorem{lemme}{Lemma}
\newtheorem{remarque}{Remark}
\newtheorem{exemple}{Example}
\newtheorem{proposition}{Proposition}
\newtheorem{corolaire}{Corollary}
\newtheorem{hyp}{Hypothesis}
\newtheorem*{rec}{Recurrence Hypothesis}
\newcommand\bel{>}
\newcommand\N{\mathbb{N}}
\newcommand\Z{\mathbb{Z}}
\newcommand\R{\mathbb{R}}
\newcommand\C{\mathbb{C}}
\newcommand\Sp{\mathbb{S}}
\newcommand\hp{\mathcal{T}^{d-1}X}
\newcommand\Tr{\mathrm{Tr}}
\newcommand\spt{\mathrm{supp}}
\title{The semi-classical scattering matrix from the point of view of Gaussian states}
\author{Maxime Ingremeau}
\affil{Laboratoire J.A. Dieudonné, Université de Nice - Sophia Antipolis}
\date{}

\maketitle

\begin{abstract}
In this paper, we will consider semiclassical scattering by compactly supported non-trapping potential on $\R^d$. We will define a family of Gaussian states on $\Sp^{d-1}$, parametrized by points in $T^*\Sp^{d-1}$, and show that the action of the scattering matrix on a Gaussian state of parameter $\rho\in T^*\Sp^{d-1}$ is still a Gaussian state, with parameter $\kappa(\rho)$, where $\kappa$ is the (classical) scattering map. This is one way of saying that \emph{the scattering matrix quantizes the scattering map}, complementary to the one introduced in \cite{Alex} in terms of Fourier Integral Operators.
\end{abstract}

\section{Introduction}
\subsection{The scattering matrix}
Consider a Schrödinger operator of the form  $P_h:= -\frac{h^2}{2} \Delta +V$ with $V\in C^\infty_c(\R^d)$. Here, $h$ is a semi-classical parameter, which will go to zero in the sequel.

It is well-known (see for instance \cite[Chapter 2]{Mel} or \cite[Chapter 3]{Resonances}) that for any $\phi_{in}\in
C^\infty (\mathbb{S}^{d-1})$, and any fixed $h>0$, there is a
unique solution to $\big{(}P_h-\frac{1}{2}\big{)}u=0$ satisfying, for all $x\in \R^d$:
\begin{equation}\label{defscattering}
u(x)=|x|^{-(d-1)/2}\big{(} e^{-i |x|/h} \phi_{in}(-\hat{x}) +
e^{i |x|/h} \phi_{out}(\hat{x}) \big{)} + O(|x|^{-(d+1)/2}),
\end{equation}
where we write $\hat{x}= \frac{x}{|x|}\in \Sp^{d-1}$.

We define the \emph{scattering matrix} $S_{h}: C^\infty (\mathbb{S}^{d-1})
\longrightarrow C^\infty (\mathbb{S}^{d-1})$, which depends on $h$, by
$$S_{h}(\phi_{in}) := e^{i\pi (d-1)/2} \phi_{out}.$$

$S_h$ may then be extended by density to an operator acting on $L^2(\Sp^{d-1})$.
The factor $e^{i\pi(d-1)/2}$ is taken so that the scattering matrix is the identity operator when $V\equiv 0$.

\begin{remarque}
The definition of the scattering matrix we took here is not exactly the standard one, but is very close to it. If we denote by $\tilde{S}_h$ the standard definition of the scattering matrix, as can be found in \cite{Mel}, and by $T$ the operator given by $(Tf)(\hat{x}) = f(-\hat{x})$, we have
$$S_h= T \tilde{S}_h T.$$
Thanks to \cite[Theorem 3.40]{Resonances}, this equality can also be written as $S_h = \tilde{S}_{-h}^{-1}$.

In particular, $S_h$ and $\tilde{S}_h$ have the same spectrum, but our definition will be more natural in relation with propagation of Gaussian wave packets.
\end{remarque}

It can be shown that $S_h$ is a unitary operator, and that $S_h-Id$ is trace class. 
The semi-classical properties of $S_h$ are closely related to the (classical) \emph{scattering map}, which we now define.

\subsection{The scattering map}
We denote by $p(x,\xi)= \frac{|\xi|^2}{2}+V(x) : T^*\R^d\longrightarrow \mathbb{R}$ the
classical Hamiltonian, which is the principal symbol of $P_h$. Let us write $\mathcal{E}$ for the energy layer of energy 1/2:
\begin{equation}\label{layer}
\mathcal{E}=\{(x,\xi)\in T^*\R^d; ~ p(x,\xi)=1/2\}.
\end{equation}

We denote by $\Phi^t(\rho)$ the Hamiltonian flow for the Hamiltonian $p$.
We will suppose in the sequel that the Hamiltonian flow is \emph{non-trapping} on the energy level $\mathcal{E}$, in the sense that
\begin{equation}\label{nontrapping}
\forall \rho\in \mathcal{E}, \exists T>0 \text{ such that } \forall t\in \R \text{ with } |t|\geq T, \text{ we have } \pi_x(\Phi^t(\rho))\notin \spt ~V.
\end{equation}
Here, $\pi_x : T^*\R^d\rightarrow \R^d$ denotes the projection on the base variable.

Since, away from $\spt ~V$, the trajectories by $\Phi^t$ are just straight lines, we have that for any $\omega\in \mathbb{S}^{d-1}$, and $\eta\in \omega^\perp\subset \mathbb{R}^d$, there
exists a unique $\rho_{\omega,\eta}\in \mathcal{E}$ such that
\begin{equation}\label{defrho}
\pi_x\big{(}\Phi^t(\rho_{\omega,\eta})\big{)} = t\omega+\eta \text{   for   } t< - T_0,
\end{equation}
 where $T_0$ is large enough so that $\spt ~V\subset B(0,T_0)$.
Here, $\omega$ is the \emph{incoming direction}, and $\eta$ is the \emph{impact parameter}. In the sequel, we will identify
\begin{equation*}
\{(\omega,\eta); ~\omega\in \mathbb{S}^{d-1}, \eta \in \omega^\perp\} \cong T^*\mathbb{S}^{d-1}.
\end{equation*} 

Thanks to the non-trapping assumption, we have that for all $\omega\in \mathbb{S}^{d-1}$, and $\eta\in \omega^\perp\subset \mathbb{R}^d$, there exists $\omega'\in \mathbb{S}^{d-1}$, $\eta'\in (\omega')^\perp\subset \mathbb{R}^d$ and $t'\in \mathbb{R}$ such that for all $t\geq T_0$,
\begin{equation*}
\pi_x\big{(} \Phi^t(\rho_{\omega,\eta})\big{)} = \omega'(t-t')+\eta'.
\end{equation*}

The \emph{(classical) scattering map} is then defined as $\kappa(\omega,\eta)=(\omega',\eta')$, as represented on Figure \ref{relation}.

\begin{figure}
    \center
   \includegraphics[scale=0.6]{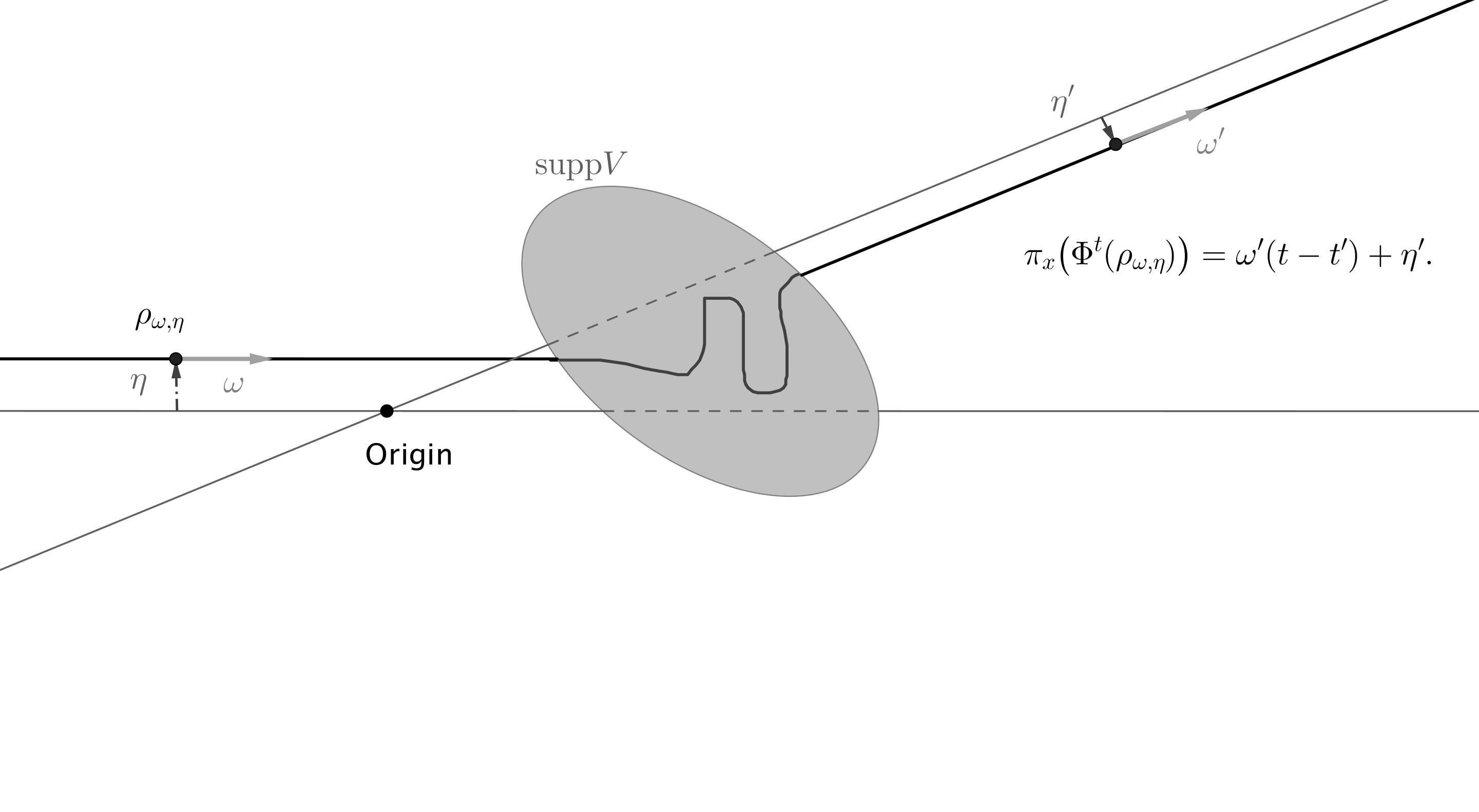}
    \caption{The scattering map $\kappa$.} \label{relation}
\end{figure}

The scattering matrix is strongly related to the scattering map in the semiclassical limit (see section \ref{Others} for some results in this direction). The aim of this paper is to express this relation in terms of \emph{Gaussian states}, which we now introduce.

\subsection{Statement of the result}
\subsubsection*{Gaussian states}
From now on, we fix a function $\chi\in C^\infty(\R; [0,1])$ such that $\chi(r) = 1$ if $r\leq 1/2$ and $\chi(r)=0$ if $r\geq 3/4$.

Let $(x_0,\xi_0) \in S^*\R^d$, let $\Gamma_0$ be a symmetric $d\times d$ matrix with positive definite real part, and $Q_0$ be a polynomial in $d$ variables. Let us write\footnote{The cut-off $\chi \Big{(} \frac{|\hat{x}-\xi|}{h^{1/3}}\Big{)}$ is here only so that the integral in (\ref{resolution2}) converges. The power $h^{-1/3}$ could be replaced by any power $h^{-\alpha}$ with $0<\alpha<1/3$.} 
$$\phi_{x_0,\xi_0,\Gamma_0,Q_0}(\hat{x};h)= \chi \Big{(} \frac{|\hat{x}-\xi|}{h^{1/3}}\Big{)} Q_0 \Big{(}\frac{\hat{x}-\xi_0}{\sqrt{h}}\Big{)} e^{-\frac{i}{h} x_0 \cdot\hat{x}}   e^{-\frac{1}{2h} (\hat{x}-\xi_0) \cdot \Gamma_0(\hat{x}-\xi_0)},$$
where we identify both $S_{x_0}^*\R^d$ and $\Sp^{d-1}$ with $\{ \hat{x}\in \R^d; |\hat{x}|=1\}$, and 
where $|\cdot|$ denotes the distance in $\R^d$. We shall say that $\phi_{x_0,\xi_0,\Gamma_0,Q_0}$ is a Gaussian state \emph{centred at $(x_0,\xi_0)$}.

Our theorem says that the image of a Gaussian state by the scattering matrix is, up to a small remainder, a Gaussian state centred at $\Phi^t(x_0,\xi_0)$ for $t$ large enough.

\begin{theoreme}\label{TheoScatGauss1}
Let $(x_0,\xi_0) \in S^*\R^d$, let $\Gamma_0$ be a symmetric $d\times d$ matrix with positive definite real part, and let $Q_0$ be a polynomial in $d$ variables with coefficients independent of $h$. 

Then there exists $(x_1,\xi_1) \in S^*\R^d$, $\delta_1\in \R$, $\Gamma_1$ a symmetric $d\times d$ matrix with positive definite real part, and polynomials $(Q_1^k)_{k\in \N}$ such that the following holds. If $N\in \N$, write $\tilde{Q}_1^N:= \sum_{0\leq k \leq N} h^{j/2} Q_1^k$. We have, for any $N\in \N$ 

\begin{equation}\label{scatteringonGaussian8}
\begin{aligned}
S_h\phi_{x_0,\xi_0,\Gamma_0,Q_0}
= e^{i\frac{\delta_1}{h}}\phi_{x_1,\xi_1,\Gamma_1,Q_1^N} + R_N,
\end{aligned}
\end{equation}
where $\|R_N\|_{C^0}= O(h^{(N+1)/2})$.

Furthermore, if we write for $i=0,1$, $(\omega_i,\eta_i):= (\xi_i, x_i- (x_i\cdot \xi_i)\xi_i)\in T^*\Sp^{d-1}$, we have
$$(\omega_1,\eta_1)= \kappa (\omega_0,\eta_0).$$
\end{theoreme}

\begin{remarque}\label{rem:patrivial}
Our proof will show that the polynomials $Q_1^k$ have degree at most $3k+ \deg(P)$.
Furthermore, the constant $\delta_1$, the polynomials $Q_1^k$ and the matrix $\Gamma_1$ depend only on $Q_0$, $\Gamma_0$, and on the values of the potential $V$ in an arbitrarily small neighbourhood of the trajectory $\{\pi_x(\Phi^t (x_1,\xi_1)); t\in \R\}$.
\end{remarque}

This result may of course be iterated, to describe the action of $S_h^k$ on a Gaussian state for any $k\in \Z$.

Our family of Gaussian states satisfies a \emph{resolution of identity} formula. Namely, if we write $\phi_{x,\xi}:= \phi_{x,\xi,Id,1}$, we will prove in section \ref{sectionreso} that, for any function $f\in L^2(\Sp^{d-1})$, we have

\begin{equation}\label{resolution2}f(\omega) = c_h \int_{\Sp^{d-1}}\mathrm{d}\xi \int_{\xi^\perp} \mathrm{d}x  \phi_{x,\xi}(\omega) \int_{\Sp^{d-1}} \mathrm{d}\omega'  \overline{ \phi_{x,\xi}(\omega')} f(\omega'),
\end{equation}
for some constant $c_h$ depending on $h$, such that $c_h\sim_{h\rightarrow 0}  2^{(d-1)/2}(2\pi h)^{-3(d-1)/2}$.

Therefore, combining (\ref{resolution2}) and Theorem \ref{TheoScatGauss1}, one may have information on the action of the scattering matrix on functions that are more general than Gaussian states. 

\begin{proof}[Sketch of proof of Theorem \ref{TheoScatGauss1}]
Let $(x_0,\xi_0) \in S^*\R^d$, let $\Gamma_0$ be a symmetric $d\times d$ matrix with positive definite real part, and $Q_0$ be a polynomial in $d$ variables. Let us write
$$\tilde{\phi}_{x_0,\xi_0,\Gamma_0,Q_0}(\hat{x};h)= Q_0 \Big{(}\frac{\xi_0-\hat{x}}{\sqrt{h}}\Big{)} e^{\frac{i}{h} x_0 \cdot\hat{x}}   e^{-\frac{1}{2h} (\hat{x}-\xi_0) \cdot \Gamma_0(\hat{x}-\xi_0)}.$$

Note that we have $\|\tilde{\phi}_{x_0,\xi_0,\Gamma_0,Q_0}(\cdot;h)- \phi_{x_0,\xi_0,\Gamma_0,Q_0}(\cdot;h)\|_{C^0} = O(h^\infty)$.
It will therefore be sufficient to prove the result with $\phi_{x_0,\xi_0,\Gamma_0,Q_0}(\cdot;h)$ replaced by $\tilde{\phi}_{x_0,\xi_0,\Gamma_0,Q_0}(\cdot;h)$.\footnote{The functions $\phi_{x_0,\xi_0,\Gamma_0,Q_0}(\cdot;h)$, which are less natural than $\tilde{\phi}_{x_0,\xi_0,\Gamma_0,Q_0}(\cdot;h)$ were actually introduced so that (\ref{resolution2}) holds. Without the cut-off $\chi$, it is not clear that the right-hand side of (\ref{resolution2}) should converge.}

We will consider on $\R^d$ a Gaussian state, of the form
\begin{equation*}
u^0_h(x)=P((x-x_0)/\sqrt{h}) e^{\frac{i}{h}x\cdot\xi_0}e^{-\frac{1}{2h}(x-x_0)\cdot \Gamma (x-x_0)},
\end{equation*}
where we suppose here for simplicity that $(x_0,\xi_0)$ is \emph{incoming}, in the sense that for all $t\leq 0$, $\pi_x(\Phi^t(x_0,\xi_0))\notin \spt~ V$. 

If we denote by $U(t)$ the Schrödinger flow generated by $P_h$,
the function $\int_{t\in \R} e^{it/2h} U(t) u^0_h  \mathrm{d}t$ will then be an eigenfunction of $P_h$. We will then try to decompose it like in (\ref{defscattering}). 

The incoming part will be given by $\int_{-\infty}^0 e^{it/2h} U(t) u^0_h  \mathrm{d}t$, which is very close to $$\int_{-\infty}^0 e^{it/2h} U_0(t) u^0_h  \mathrm{d}t,$$ where $U_0$ is the free Schrödinger flow. In section \ref{sectionfree}, we will study the behaviour of $\int_{-\infty}^0 e^{it/2h} U_0(t) u^0_h  \mathrm{d}t$ as $|x|\rightarrow \infty$.

To analyse the outgoing part, given by $\int_{0}^{\infty} e^{it/2h} U(t) u^0_h  \mathrm{d}t$, we will use the fact that, when a Gaussian state is propagated by the Schrödinger flow, it remains a Gaussian state (see section \ref{Reminder} for more details). We will use this to propagate $u_h^0$ until a time $T>0$ such that $\Phi^T(x_0,\xi_0)$ is outgoing. The integral $\int_{T}^{\infty} e^{it/2h} U(t) u^0_h  \mathrm{d}t$ is then very close to $\int_{T}^{\infty} e^{it/2h} U_0(t) U(T) u^0_h  \mathrm{d}t$. We will then use the results of section \ref{sectionfree} to describe the behaviour of this integral as $|x|\rightarrow \infty$. These asymptotics will give us the result of Theorem \ref{TheoScatGauss1}.
\end{proof}

\subsection{Relation to other works}\label{Others}
The study of the semi-classical properties of the scattering matrix, and its links with the scattering map has a long history. The scattering amplitude, that is to say, the integral kernel of $S_h-Id$ was expanded as a semi-classical series in different situations in \cite{majda1976high}, \cite{guillemin1977sojourn},
\cite{vainberg1977quasiclassical}, 
\cite{protas1982quasiclassical},  \cite{robert1987semi},\cite{yajima1987quasi}, and
\cite{michel2004semi}.

It was shown in \cite{Alex} that, microlocally near non-trapped points, the scattering matrix\footnote{Although not stated explicitly, the convention for the definition of the scattering matrix in \cite{Alex} is the same as ours, i.e., $S_h$ and not $\tilde{S}_h$} $S_h$ is a \emph{Fourier Integral Operator} quantizing the scattering map $\kappa$. This result was extended to short range potentials in \cite{alexandrova2006structure}, and to non-trapping asymptotically conical manifolds in \cite{HaWu}.

These results were then used in \cite{EquidSph}, \cite{Equid}, \cite{EquidPol}, and \cite{ingremeau2016equidistribution} to obtain some equidistribution results on the spectrum of $S_h$.

Although the results presented here could be deduced from those of \cite{Alex}, the method used here is quite different, and somehow simpler and more explicit. In particular, it is not clear how one could deduce from Alexandrova's results that the image of a Gaussian state by the scattering matrix depends only on the values of the potential close to a single trajectory, as stated in Remark \ref{rem:patrivial}.

Our approach, based on the propagation of coherent states, is largely inspired by \cite{robert2007propagation} and \cite{combescure2012coherent}.
Actually, in \cite{robert2007propagation} and \cite{combescure2012coherent}, the action of the \emph{scattering operator} on Gaussian states is described precisely. The scattering matrix can be seen as the restriction of the scattering operator to an energy shell. However, it does not seem easy to deduce our result from theirs, and the approach presented below, based on asymptotics as $|x|\rightarrow\infty$, seems more direct.

Though we only considered the simplest case of potential scattering, the result of the present paper also hold for compactly supported metric perturbations of the Laplacian, with a similar proof . Our results should still hold if the trapped set is non-empty, as long as the cut-off resolvent is bounded by some negative power of $h$, and that the Gaussian state we consider is centred on a point which is not trapped in the future. Actually, the methods employed here should in principle work as long as we consider Gaussian states centred at some points which spend less than the Ehrenfest time in the interaction region; this could, in some situations (for example, if the trapped set is hyperbolic), allow us to describe the action of the scattering matrix outside of a neighbourhood of the trapped set whose size depends on $h$. If the trapped set is hyperbolic, and the topological pressure of half the Jacobian flow is negative, it should be possible to combine the ideas of the present note with the methods developed in \cite{NZ} and  \cite{Ing} to describe the scattering matrix microlocally near the trapped set. These issues will be pursued elsewhere.

\paragraph{Acknowledgements} The author would like to thank the anonymous referees for their many comments, which greatly helped to improve the presentation of this paper.

The author was funded by the LabEx IRMIA, and partially supported by the
Agence Nationale de la Recherche project GeRaSic (ANR-13-BS01-0007-01).
\section{Free evolution of a Gaussian wave packet times a polynomial}\label{sectionfree}
Let $(x_0,\xi_0)\in S^*\R^d$. Let $P$ be a polynomial in $x$ of degree $n$, and let $\Gamma$ be a symmetric $d\times d$ matrix with positive definite real part.
We consider on $\mathbb{R}^d$ the function

\begin{equation}\label{GaussienGeneral}
u^0_h(x):=P\Big{(}\frac{x-x_0}{\sqrt{h}}\Big{)} e^{\frac{i}{h}x\cdot\xi_0}e^{-\frac{1}{2h}(x-x_0)\cdot \Gamma (x-x_0)}.
\end{equation}

We shall denote by $U_0= e^{i t h\Delta/2}$ the free Schrödinger evolution. Fix $T>0$. We want to consider $$v_h:=\int_{T}^{+\infty} U_0(t) u^0_h e^{it/2h} \mathrm{d}t.$$

More precisely, the aim of this section is to compute 
$$\lim \limits_{|x|\rightarrow \infty} |x|^{(d-1)/2} e^{-i|x|/h} v_h( \hat{x} |x|).$$

\paragraph{Considerations on the Fourier transform of a Gaussian times a polynomial}
We take the following convention for the Fourier transform:
$$\mathcal{F} \phi(\xi) := \int_{\R^d} e^{-i x\cdot \xi} \phi(x) \mathrm{d}x.$$

Consider an invertible symmetric matrix $\Gamma\in M_d(\C)$ with positive definite real part, and a polynomial $P$ of $d$ variables and of degree $n$. 
Then there exists a unique polynomial of degree $n$, which we shall denote by $P_\Gamma$ such that $\mathcal{F}\big{(} P(x) e^{-\frac{x\cdot \Gamma x}{2}}\big{)}(\xi) = P_\Gamma(\xi) e^{-\frac{\xi\cdot \Gamma^{-1}\xi}{2}}$.

For any $\Gamma$, the map $P\mapsto P_\Gamma$ is an automorphism of the space of polynomials of degree $n$. The image of the polynomial constant equal to 1 is the constant 
\begin{equation}\label{imageun}
\frac{(2\pi)^{d/2}}{\sqrt{\det \Gamma}},
\end{equation}
where $A\mapsto \sqrt{\det A}$ is the unique analytic branch satisfying $\sqrt{\det A}>0$ when $A$ is real (cf. \cite[\S 3.4]{hormander1990analysis}).

In the sequel, we will need an asymptotic for $P_{\Gamma+it} (t\xi)$ as $t\rightarrow +\infty$.
By definition, we have:
\[\begin{aligned}
P_{\Gamma+it}(t\xi) e^{-\frac{t^2}{2} \xi \cdot (\Gamma+it Id)^{-1} \xi} &=\int_{\R^d} P(x) e^{-\frac{1}{2} x \cdot (\Gamma+it Id) x} e^{-it x\cdot \xi} \mathrm{d}x\\
&= \int_{\R^d} P(x) e^{-\frac{1}{2} x \cdot \Gamma x} e^{-it (|x|^2/2+ x\cdot \xi)} \mathrm{d}x.
\end{aligned}\]

We can use stationary phase here, with $1/t$ as a small parameter. The phase $x\cdot \xi + |x|^2/2$ is stationary at $x=-\xi$, so that we have
\[\begin{aligned}
P_{\Gamma+it}(t\xi) e^{-\frac{t^2}{2} \xi \cdot (\Gamma+it Id)^{-1} \xi} &= e^{-di\pi/4}P(-\xi) e^{-\frac{1}{2}\xi \cdot \Gamma \xi} e^{\frac{1}{2}it|\xi|^2} \Big{(}\frac{2\pi}{t}\Big{)}^{d/2} + O( t^{-d/2-1}). 
\end{aligned}\]

On the other hand, we have  
\begin{equation}\label{DLMatrix}
(\Gamma+it Id)^{-1}= -\frac{i}{t}Id + \frac{1}{t^2}\Gamma + O\Big{(}\frac{1}{t^3}\Big{)},
\end{equation}
 so that
$$e^{-\frac{t^2}{2} \xi \cdot (\Gamma+it Id)^{-1} \xi}= e^{i\frac{t}{2} |\xi|^2} e^{-\frac{1}{2} \xi \Gamma \xi} + O\Big{(}\frac{1}{t}\Big{)}.$$

Hence 
\begin{equation}\label{equivpoly}
\begin{aligned}
P_{\Gamma+it}(t\xi) &= e^{-di\pi/4}P(-\xi)\Big{(}\frac{2\pi}{t}\Big{)}^{d/2} + O( t^{-d/2-1}). 
\end{aligned}
\end{equation}

\paragraph{Semiclassical Fourier transform}

We define the semiclassical Fourier transform by $$\mathcal{F}_h \phi(\xi) := \int_{\R^d} e^{-\frac{i}{h} x\cdot \xi} \phi(x) \mathrm{d}x,$$
and its inverse is given by 
$$\mathcal{F}_h^{-1} \psi(x) :=\frac{1}{(2\pi h)^{d}} \int_{\R^d} e^{\frac{i}{h} x\cdot \xi} \psi(\xi) \mathrm{d}\xi.$$

Note that, for any function $f\in L^2(\R^d)$, we have 
\begin{equation}\label{eq:Rescaling}
\mathcal{F}_h (f(\cdot/\sqrt{h}))(\xi) = h^{d/2} \big{(}\mathcal{F}(f)\big{)}(\xi/\sqrt{h}),
\end{equation}
 and that for any function $f\in H^2(\R^d)$, we have
$\mathcal{F}_h (-h^2 \Delta f ) = \xi^2 \mathcal{F}_h (f)$.

\paragraph{Free evolution}

The semiclassical Fourier transform of $u^0_h$ is given by
\[\begin{aligned}
\mathcal{F}_h u^0_h(\xi) &= h^{d/2}  e^{\frac{i}{h} x_0\cdot \xi_0} P_\Gamma\Big{(}\frac{\xi-\xi_0}{\sqrt{h}}\Big{)} e^{-\frac{i}{h} x_0\cdot \xi} e^{-\frac{1}{2h}(\xi-\xi_0)\cdot \Gamma^{-1}(\xi-\xi_0)}.
\end{aligned}\]

Therefore 
\begin{equation*}
\begin{aligned}
U_0(t) u^0_h(x) &= \mathcal{F}_h^{-1} \Big{(} e^{-it\xi^2/2h} P_\Gamma\Big{(}\frac{\xi-\xi_0}{\sqrt{h}}\Big{)} e^{\frac{i}{h} x_0\cdot (\xi_0-\xi)} e^{-\frac{1}{2h}(\xi-\xi_0)\cdot \Gamma^{-1}(\xi-\xi_0)} \Big{)}(x)\\
&= \frac{1}{(2\pi)^d h^{d/2}}  \int_{\R^d}  e^{-it\xi^2/2h} P_\Gamma\Big{(}\frac{\xi-\xi_0}{\sqrt{h}}\Big{)} e^{\frac{i}{h} x_0\cdot (\xi_0-\xi)} e^{-\frac{1}{2h}(\xi-\xi_0)\cdot \Gamma^{-1}(\xi-\xi_0)} e^{i x\cdot \xi/h} \mathrm{d}\xi\\
&=\frac{1}{(2\pi)^d h^{d/2}} e^{\frac{i}{h} x\cdot \xi_0}\\
&\int_{\R^d} P_\Gamma\Big{(}\frac{\xi'}{\sqrt{h}}\Big{)} e^{\frac{i}{h} (x-x_0)\cdot \xi'} e^{-it \xi'^2/2h} e^{\frac{-it}{h}\xi' \cdot \xi_0} e^{-it \xi_0^2/2h}  e^{-\frac{1}{2h}\xi'\cdot \Gamma^{-1} \xi' }  \mathrm{d}\xi' ~~\text{ by setting } \xi= \xi'+\xi_0\\
&=\frac{1}{(2\pi)^d h^{d/2}} e^{\frac{i}{h} x\cdot \xi_0} e^{-it/2h}\int_{\R^d} P_\Gamma\Big{(}\frac{\xi'}{\sqrt{h}}\Big{)} e^{\frac{i}{h} (x-x_0-t\xi_0)\cdot \xi'}  e^{-\frac{1}{2h}\xi'\cdot (\Gamma^{-1}+it Id) \xi' }  \mathrm{d}\xi'.
\end{aligned}
\end{equation*}
By (\ref{eq:Rescaling}), we obtain that 
\begin{equation}\label{freeevo}
\begin{aligned}
U_0(t) u^0_h(x)= \frac{1}{(2\pi)^d} e^{\frac{i}{h} x\cdot \xi_0} e^{-it/2h}(P_{\Gamma})_{\Gamma^{-1}+it Id} \Big{(}-\frac{x-x_0-t\xi_0}{\sqrt{h}}\Big{)}  e^{-\frac{1}{2h} (x-x_0-t\xi_0)\cdot (\Gamma^{-1}+it Id)^{-1} (x-x_0-t\xi_0)}.
\end{aligned}
\end{equation}

\paragraph{Behaviour at infinity}

Consider the integral $v_h=\int_0^\infty U_0(t) u^0_h e^{it/2h} \mathrm{d}t.$
Thanks to the computations in the previous paragraph, we see that this integral converges on any compact set of $\R^d$. We want to compute $\lim \limits_{|x|\rightarrow \infty} |x|^{(d-1)/2} e^{-i|x|/h} v_h( \hat{x} |x|)$.

Note that, in what follows, $h$ will only be a parameter, which will play no special role.
 We have

\[\begin{aligned}
 &v_h( \hat{x} |x|)= \\
& \frac{1}{(2\pi)^d} e^{\frac{i}{h} x\cdot \xi_0} \int_{t=T}^\infty (P_{\Gamma})_{\Gamma^{-1}+it Id} \Big{(}-\frac{x-x_0-t\xi_0}{\sqrt{h}}\Big{)} e^{-\frac{1}{2h} (x-x_0-t\xi_0)\cdot (\Gamma^{-1}+it Id)^{-1} (x-x_0-t\xi_0)}\mathrm{d}t\\
&=\frac{1}{(2\pi)^d} e^{\frac{i}{h} x\cdot \xi_0} \\&\int_{\tau=\frac{T}{|x|}}^\infty (P_{\Gamma})_{\Gamma^{-1}+i |x|\tau Id} \Big{(}-\frac{|x|}{\sqrt{h}}(\hat{x}-x_0/|x|-\tau\xi_0)\Big{)} e^{-\frac{|x|^2}{2h} (\hat{x}-x_0/|x|-\tau\xi_0)\cdot (\Gamma^{-1}+i|x|\tau Id)^{-1} (\hat{x}-x_0/|x|-\tau\xi_0)} |x|\mathrm{d}\tau,
\end{aligned}\]
where we obtained the last line by setting $t= |x| \tau$.

Note that for $\tau$ not close to one, the above integrand is a $O(|x|^{-\infty})$.
$h$ being fixed, let us compute an asymptotic expansion of $e^{-\frac{|x|^2}{2h} (\hat{x}-x_0/|x|-\tau\xi_0)\cdot (\Gamma^{-1}+i|x|\tau Id)^{-1} (\hat{x}-x_0/|x|-\tau\xi_0)}$ as $|x|\rightarrow \infty$, for $\tau \geq 1/2$. We have, using (\ref{DLMatrix}) that
\[\begin{aligned}
&|x|^2 (\hat{x}-x_0/|x|-\tau\xi_0)\cdot (\Gamma^{-1}+i|x|\tau Id)^{-1} (\hat{x}-x_0/|x|-\tau\xi_0)\\
 &= - \frac{i|x|}{\tau} |\hat{x}-x_0/|x|-\tau\xi_0|^2 + \frac{1}{\tau^2} (\hat{x}-x_0/|x|-\tau\xi_0) \cdot \Gamma^{-1}(\hat{x}-x_0/|x|-\tau\xi_0) + O \Big{(}\frac{1}{|x|}\Big{)}\\
&= - \frac{i|x|}{\tau} |\hat{x}-\tau\xi_0|^2 +\frac{2i}{\tau} x_0\cdot(\hat{x}-\tau \xi_0) +  \frac{1}{\tau^2} (\hat{x}-\tau\xi_0) \cdot \Gamma^{-1}(\hat{x}-\tau\xi_0) + O \Big{(}\frac{1}{|x|}\Big{)} .
\end{aligned}\]

Therefore, 
\[\begin{aligned}
&v_h(|x| \hat{x})\sim \frac{1}{(2\pi)^d} e^{\frac{i}{h} x\cdot \xi_0}\\
&\times \int_{1/2}^{+\infty} e^{i\frac{|x|}{2h\tau }|\hat{x}-\tau \xi_0|^2} (P_{\Gamma})_{\Gamma^{-1}+i |x|\tau Id} \Big{(}-\frac{|x|}{\sqrt{h}}(\hat{x}-\tau\xi_0)\Big{)} e^{\frac{1}{2h} (-\frac{2i}{\tau} x_0\cdot(\hat{x}-\tau \xi_0) +  \frac{1}{\tau^2} (\hat{x}-\tau\xi_0) \cdot \Gamma^{-1}(\hat{x}-\tau\xi_0))} |x|\mathrm{d}\tau.
\end{aligned}\]

This can be seen as a stationary phase, with $|x|^{-1}$ as a small parameter. The phase is, up to a factor $1/(2h)$, $\phi(\tau) = |\hat{x}-\tau \xi_0|^2/\tau$. By an easy computation, we get 
$$\phi'(\tau)= -\frac{1}{\tau^2} (\hat{x}-\tau \xi_0)\cdot (\hat{x}+\tau \xi_0) = -\frac{1}{\tau^2} (1-\tau^2 )= 1-\frac{1}{\tau^2},$$
which vanishes if and only if $\tau=1$.
We have $\phi''(\tau) =  \frac{2}{\tau^3}$, so that $\phi''(1)= 2$.

By the stationary phase formula, we get that 
\[\begin{aligned} 
&|x|^{(d-1)/2} e^{-i|x|/h} v_h( \hat{x} |x|) \\
&\sim \frac{1}{(2\pi)^d}  |x|^{(d-1)/2} e^{-i|x|/h}e^{i\pi/4}\pi^{1/2} |x|^{-1/2} e^{\frac{i}{h} x\cdot \xi_0} e^{i \frac{|x|}{2h} |\hat{x}-\xi_0|^2} \\
&\times (P_{\Gamma})_{\Gamma^{-1}+i |x| Id} \Big{(}-\frac{|x|}{\sqrt{h}}(\hat{x}-\xi_0)\Big{)} e^{-\frac{1}{2h} (2i x_0\cdot(\hat{x}-\xi_0) +  (\hat{x}-\xi_0) \cdot \Gamma^{-1}(\hat{x}-\xi_0))} |x|.
\end{aligned}
\]
Thanks to (\ref{equivpoly}), we obtain that
\[\begin{aligned} 
&|x|^{(d-1)/2} e^{-i|x|/h} v_h( \hat{x} |x|) \\
&\sim \frac{1}{(2\pi)^d} |x|^{(d-1)/2}  e^{-i|x|/h} \pi^{1/2} |x|^{1/2} e^{\frac{i}{h} x\cdot \xi_0} e^{\frac{i|x|}{2h} |\hat{x}-\xi_0|^2} P_\Gamma\Big{(}{\frac{\hat{x}-\xi_0}{\sqrt{h}}}\Big{)}  \\
&\times e^{i(1-d)\pi/4} \Big{(}\frac{2\pi}{|x|}\Big{)}^{d/2}  e^{-\frac{1}{2h} (2i x_0\cdot(\hat{x}-\xi_0) +  (\hat{x}-\xi_0) \cdot \Gamma^{-1}(\hat{x}-\xi_0)} \\
&\sim e^{i(1-d)\pi/4}\frac{\pi^{1/2}}{(2\pi)^{d/2}}  P_\Gamma\Big{(}{\frac{\hat{x}-\xi_0}{\sqrt{h}}}\Big{)} e^{-\frac{i}{h} x_0\cdot(\hat{x}-\xi_0)}  e^{-\frac{1}{2h} (\hat{x}-\xi_0) \cdot \Gamma^{-1}(\hat{x}-\xi_0))} e^{\frac{i|x|}{2h} (2\hat{x}\cdot \xi_0 - 2 +| \hat{x}-\xi_0|^2)}\\
&\sim e^{i(1-d)\pi/4}\frac{\pi^{1/2}}{(2\pi)^{d/2}} P_\Gamma\Big{(}{\frac{\hat{x}-\xi_0}{\sqrt{h}}}\Big{)} e^{-\frac{i}{h} x_0\cdot (\hat{x}-\xi_0)}  e^{-\frac{1}{2h} (\hat{x}-\xi_0) \cdot \Gamma^{-1}(\hat{x}-\xi_0)}.
\end{aligned} \]

All in all, we have obtained that, for any $T\in \R$, we have
\begin{equation}\label{futur}
\begin{aligned}
&\lim\limits_{|x|\rightarrow \infty} |x|^{(d-1)/2} e^{-i|x|/h} \Big{(}\int_{T}^{+\infty} U_0(t) u^0_h e^{it/2h} \mathrm{d}t\Big{)}(|x|\hat{x})  \\
&= e^{i(1-d)\pi/4}\frac{\pi^{1/2}}{(2\pi)^{d/2}} P_\Gamma\Big{(}{\frac{\hat{x}-\xi_0}{\sqrt{h}}}\Big{)} e^{-\frac{i}{h} x_0\cdot(\hat{x}-\xi_0)}  e^{-\frac{1}{2h} (\hat{x}-\xi_0) \cdot \Gamma^{-1}(\hat{x}-\xi_0)}.
\end{aligned}
\end{equation}

On the other hand, recall that for any $f\in L^2(\R^d)$, we have $\big{(}U_0(-t) f\big{)}(x) = \overline{\big{(}U_0(t) \overline{f}\big{)}(x)} $, so that
\begin{equation*}
\begin{aligned} 
&\lim\limits_{|x|\rightarrow \infty} |x|^{(d-1)/2} e^{i|x|/h} \Big{(}\int_{-\infty}^{T} U_0(t) u^0_h e^{it/2h} \mathrm{d}t\Big{)}(|x|\hat{x})\\
&= \overline{\lim\limits_{|x|\rightarrow \infty} |x|^{(d-1)/2} e^{-i|x|/h} \Big{(}\int_{-T}^{+\infty} U_0(t) \overline{u^0_h} e^{it/2h} \mathrm{d}t\Big{)}(|x|\hat{x})}.
\end{aligned} 
\end{equation*}
Now, since $\overline{u^0_h}$ is of the form (\ref{GaussienGeneral}) with $\xi_0$ replaced by $-\xi_0$ and $\Gamma$ replaced by $\overline{\Gamma}$, (\ref{futur}) gives us  

\begin{equation}\label{passe}
\begin{aligned} 
&\lim\limits_{|x|\rightarrow \infty} |x|^{(d-1)/2} e^{i|x|/h} \Big{(}\int_{-\infty}^{T} U_0(t) u^0_h e^{it/2h} \mathrm{d}t\Big{)}(|x|\hat{x}) \\
&= e^{i(d-1)\pi/4} \frac{\pi^{1/2}}{(2\pi)^{d/2}} P_\Gamma\Big{(}{-\frac{\xi_0+\hat{x}}{\sqrt{h}}}\Big{)} e^{\frac{i}{h} x_0\cdot(\hat{x}+\xi_0)}  e^{-\frac{1}{2h} (\hat{x}+\xi_0) \cdot \overline{\Gamma}^{-1}(\hat{x}+\xi_0)}.
\end{aligned} 
\end{equation}
\section{Construction of a generalized eigenfunction}
Let $(x_0,\xi_0)\in S^*\R^d$. Let $P$ be a polynomial in $x$ of degree $n$, and let $\Gamma$ be a symmetric $d\times d$ matrix with positive definite real part.
As in the previous section, we consider on $\mathbb{R}^d$ the function

\begin{equation}\label{GaussienGeneral2}
u^0_h(x):=P\Big{(}\frac{x-x_0}{\sqrt{h}}\Big{)} e^{\frac{i}{h}x\cdot\xi_0}e^{-\frac{1}{2h}(x-x_0)\cdot \Gamma (x-x_0)}.
\end{equation}

Since $V$ has compact support, we may find a time $t_-$ such that for all $t\geq t_-$, $x_0-t\xi_0\notin \spt ~V$. We shall write $(x_-,\xi_-):= (x_0-t_-\xi_0,\xi_0)$.

Let us consider
\begin{equation*}
u^-_h:= U_0(t_-) u_h^0.
\end{equation*}

By (\ref{freeevo}), $u_h^-$ may be put in the form
\begin{equation}\label{GaussienGeneral6}
u^-_h(x):=P_-\Big{(}\frac{x-x_-}{\sqrt{h}}\Big{)} e^{-i t_-/(2h)} e^{\frac{i}{h}x\cdot\xi_-}e^{-\frac{1}{2h}(x-x_-)\cdot \Gamma_- (x-x_-)},
\end{equation}
where $P_-$ is a polynomial in $x$ of degree $n$, and let $\Gamma_-$ is a symmetric $d\times d$ matrix with positive definite real part.

Since the Hamiltonian has been supposed non-trapping, we may find a time $t_+$ such that $\pi_x\big{(} \Phi^{t_++t}(x_-,\xi_-)\big{)}$ is not in $\spt~ V$ for any $t\geq 0$. We shall write $(x_+,\xi_+):= \Phi^{t_+}(x_-,\xi_-)$.

In particular, for any $t\geq 0$, we have $\Phi^{t_+ +t}(x_-,\xi_-) = (x_+ +t \xi_+, \xi_+)$. Recalling the definition of the scattering map, we see that, if we write $x_\pm^\perp:=x_\pm -(x_\pm \cdot\xi_\pm)\xi_\pm $, then 
\begin{equation}\label{diffusionclassique}
 (\xi_+, x_+^\perp) = \kappa (\xi_-,x_-^\perp),
\end{equation}
where we identified $(\xi_\pm, x_\pm^\perp )$ with points in $T^*\mathbb{S}^{d-1}$.

We denote by $U(t)= e^{-\frac{it}{h} P_h}$ the Schrödinger propagator. We set
\begin{equation}\label{forward}
u^+_h:= U (t_+) u^-_h.
\end{equation}

Let us recall how the results from \cite{robert2007propagation} can be used to describe $u^+_h$ in the semi-classical limit.

\subsection{Review of the propagation of a Gaussian times a polynomial}\label{Reminder}

In \cite[Theorem 0.1]{robert2007propagation}, it is shown that for any $t\in \mathbb{R}$, for any $N\in \N$, and for any function of the form (\ref{GaussienGeneral6}), we have
\begin{equation}\label{propagGauss}
U(t)u^-_h (x)= \tilde{u}_h (t,x;N) + R_h^N(t,x),
\end{equation}
with
\begin{equation*}
\tilde{u}_h(t,x;N) = e^{i\frac{\delta_t}{h}} \sum_{0\leq j\leq N} h^{j/2} \pi_j\Big{(}t, \frac{x- x_t}{\sqrt{h}}\Big{)} e^{\frac{i}{h}x\cdot\xi_t} e^{-\frac{1}{2h}(x-x_t)\cdot \Gamma_t (x-x_t)},
\end{equation*}
where
\begin{itemize}
\item $z_t=(x_t,\xi_t)= \Phi^t(x_-,\xi_-)$.
\item $\Gamma_t$ is a symmetric complex matrix with positive definite real part.
More precisely, if we write
\begin{equation*}
A_t:= \frac{\partial x_t}{\partial x_-},~~B_t:= \frac{\partial \xi_t}{\partial x_-},~~ C_t:= \frac{\partial x_t}{\partial \xi_-},~~ D_t:= \frac{\partial \xi_t}{\partial \xi_-},
\end{equation*}
we have $$\Gamma_t= (C_t+iD_t\Gamma)(A_t+iB_t\Gamma)^{-1}.$$
\item The $\pi_j$ are polynomials of degree at most $3j + \deg(P)$, and 
\begin{equation}\label{enzero}
\pi_-(t)= P(0) (\det (A_t+iB_t \Gamma))^{1/2}.
\end{equation}
\item $\delta_t$ is the action integral
$$\delta_t= \int_0^t (\frac{|\xi_s|^2}{2}-V(x_s)) \mathrm{d}s - \frac{x_t\cdot\xi_t + x_-\cdot \xi_-}{2} = \int_0^t x_s\cdot \nabla V(x_s) \mathrm{d}s + \frac{x_t\cdot\xi_t - x_-\cdot \xi_-}{2}-t.$$
Note that, if $x_t\notin  \spt ~ V$, we have $\delta'(t)=0$. 

\item For any $\alpha\in \mathbb{N}$, there exists $C_\alpha(t)>0$ such that
\begin{equation}\label{fastdecay}
\big{\|} (1+|x|)^\alpha R_h^N(t,x)\big{\|}_{L^2}\leq C_\alpha(t) h^{(N+1)/2}.
\end{equation}
\end{itemize}

Thanks to this, we see that the function $u^+_h$ defined in (\ref{forward}) may be put in the form
$$u^+_h (x) = \tilde{u}^+_h (x)+ R_h^N(x),$$
where
\begin{equation}\label{propagGauss2}
\tilde{u}^+_h (x)= e^{i\frac{\delta_+}{h}} \sum_{0\leq j\leq N} h^{j/2} \pi^+_j\Big{(}\frac{x- x_+}{\sqrt{h}}\Big{)} e^{\frac{i}{h}x\cdot\xi_+} e^{-\frac{1}{2h}(x-x_+)\cdot \Gamma_+ (x-x_+)} ,
\end{equation}
and where $R_h^N(x)$ satisfies (\ref{fastdecay}).

\subsection{Construction of a generalized eigenfunction}
Let us consider
$$E^0_h:=\int_{-\infty}^{0} U_0(t) u^-_h e^{it/2h} \mathrm{d}t+\int_{0}^{t_+} U(t) u^-_h e^{it/2h} \mathrm{d}t+\int_{0}^{+\infty} U_0(t) \tilde{u}^+_h e^{it/2h} \mathrm{d}t.$$

Thanks to the computations in section \ref{sectionfree}, we see that for any compact set $\mathcal{K}$ and any $k\in \N$, $\| U_0(t) u^-_h\|_{C^k(\mathcal{K})}$ decays exponentially fast as $t\rightarrow -\infty$, so that $\int_{-\infty}^{0} \|U_0(t) u^-_h e^{it/2h}\|_{C^k(\mathcal{K})} \mathrm{d}t$ converges.

Thanks to (\ref{freeevo}), we may describe $ U_0(\pm t) \tilde{u}_h^\pm$ for all $t\geq 0$,  and we also have thanks to the computations in section \ref{sectionfree} that $\int_{0}^{+\infty} \| U_0(t) \tilde{u}^+_h e^{it/2h}\|_{C^k(\mathcal{K})} \mathrm{d}t$ converges. Therefore, the function $E_h^0$ is well-defined, and belongs to $C^\infty(\R^d)$.

If $f\in H^2(\R^d)$, we have 
\begin{align*}
i h \frac{\mathrm{d}}{\mathrm{d} t} \big{[} e^{it/2h} U_0(t) f \big{]} &= \big{(}-\frac{\Delta}{2} - \frac{1}{2}\big{)} e^{it/2h} U_0(t) f\\
i h \frac{\mathrm{d}}{\mathrm{d} t} \big{[} e^{it/2h} U(t) f \big{]} &= \big{(}P_h- \frac{1}{2}\big{)} e^{it/2h} U(t) f.
\end{align*}

Therefore, differentiating $E_h^0$ under the integral signs, we obtain
\begin{equation*}
\begin{aligned}
(P_h-1/2) E^0_h &= \int_{-\infty}^{0} (P_h-1/2) U_0(t) u^-_h e^{it/2h} \mathrm{d}t+\int_{0}^{t_+} (P_h-1/2) U(t) u^-_h e^{it/2h} \mathrm{d}t\\
&+\int_{0}^{+\infty} (P_h-1/2) U_0(t) \tilde{u}^+_h e^{it/2h} \mathrm{d}t\\
&=  \int_{-\infty}^{0} \big{[}ih \frac{\mathrm{d}}{\mathrm{d}t}+V\big{]}  U_0(t) u^-_h e^{it/2h} \mathrm{d}t+\int_{0}^{t_+} ih\frac{\mathrm{d}}{\mathrm{d}t} U(t) u^-_h e^{it/2h} \mathrm{d}t\\
&+\int_{0}^{+\infty} \big{[}ih\frac{\mathrm{d}}{\mathrm{d}t}+V\big{]} U_0(t) \tilde{u}^+_h e^{it/2h} \mathrm{d}t\\
&= \int_{-\infty}^{0} V U_0(t) u^-_h e^{it/2h} \mathrm{d}t+ ih u_h^- + ih (u_h^+-u_h^-)- ih\tilde{u}_h^+ + \int_{0}^{+\infty} V U_0(t) \tilde{u}^+_h e^{it/2h} \mathrm{d}t\\
&= \int_{-\infty}^{0} V U_0(t) u^-_h e^{it/2h} \mathrm{d}t+\int_{0}^{+\infty} V U_0(t) \tilde{u}^+_h e^{it/2h} \mathrm{d}t + R_h^N(x).
\end{aligned}
\end{equation*}

Thanks to (\ref{freeevo}), we may describe $ U_0(-t) u_h^-$ and $U_0(t) \tilde{u}_h^+$ for all $t\geq 0$. They are polynomials times Gaussian functions centred at $x_{\pm} \pm t \xi_\pm$, hence they are exponentially small on the support of the potential $V$.

More precisely, there exists $C,c>0$ such that for any $t\geq 0$, we have 
\begin{equation}\label{freedecay}
\|V U_0(- t) u^-_h\|_{L^2} \leq Ce^{-(c+t)/h}, ~~ \|V U_0(t) \tilde{u}^+_h\|_{L^2} \leq Ce^{-(c+t)/h}.
\end{equation}

Therefore, $\int_{-\infty}^{0} V U_0(t) u^-_h e^{it/2h} \mathrm{d}t+\int_{0}^{+\infty} V U_0(t) \tilde{u} ^+_h e^{it/2h} \mathrm{d}t $ is compactly supported with support in $\spt ~V$, and has $L^2$ norm which is $O(h^\infty)$. Since $R_h^N$ satisfies (\ref{fastdecay}), we deduce that for any $\alpha\in \mathbb{N}$, there exists $C_\alpha>0$ such that
\begin{equation*}
\big{\|} (1+|x|)^\alpha (P_h-1/2)E_h^0 \big{\|}_{L^2}\leq C_\alpha h^{(N+1)/2}.
\end{equation*}

We define the semi-classical outgoing resolvent as $R_h= (P_h+(1/2+i0))^{-1}$. It is well-defined on any function in the weighted space $(1+|x|)^{-1} L^2$, and thanks to the non-trapping assumption, we have by \cite[Theorem 2]{robert1987semi} that
\begin{equation}\label{resolvestim}
\big{\|} (1+|x|)^{-1} R_h (1+|x|)^{-1} \big{\|}_{L^2\rightarrow L^2} \leq \frac{C}{h}.
\end{equation}

Write $E_h^1:= R_h (P_h-1/2) E^0_h$. We have
$$ \big{\|} (1+|x|)^{-1} E_h^1 \big{\|}_{L^2} \leq C h^{(N-1)/2}.$$

The function
$$E_h:= E^0_h + E_h^1,$$
which we call a \emph{distorted Gaussian beam}, is bounded polynomially in $|x|$ for every $h$, and it satisfies 
$$(P_h-1/2) E_h = 0.$$

We will now show that it can be written as 
$$
E_h(x)=|x|^{-(d-1)/2}\big{(} e^{-i |x|/h} E_h^{in}(-\hat{x}) +
e^{i |x|/h} E_h^{out}(\hat{x}) \big{)} + O(|x|^{-(d+1)/2}),
$$
where $E_h^{in}$ and $E_h^{out}$ are smooth functions which we will identify.

\subsection{Behaviour of $E_h$ near infinity}
The following lemma guaranties that the term $E_h^1$ is negligible in our computations.
 
 \begin{lemme}
 There exists a function $a_h \in C^\infty(\Sp^{d-1})$ with $\|a_h\|_{L^2}= O(h^{N/2-1})$ such that
 $$E_h^1 (|x|\hat{x}) \sim_{|x|\rightarrow \infty} |x|^{(d-1)/2} e^{i |x|/h} a_h(\hat{x}).$$
 \end{lemme}
 \begin{proof}
 Thanks to \cite[Proposition 2.4]{Mel}, we know that $\lim\limits_{|x|\rightarrow \infty} |x|^{-(d-1)/2} e^{-i |x|/h} E_h^1(|x|\hat{x})$ exists, and is equal to $\frac{1}{2ih}\mathcal{P}_{h,V}^* E_h^1$, where $\mathcal{P}_{h,V}^*$ is the adjoint of the Poisson operator $\mathcal{P}_{h,V} : L^2(\Sp^{d-1})\rightarrow L^2(\R^d)$. 
 
We have
\begin{equation*}
\begin{aligned}
\|\mathcal{P}_{h,V}^* E_h^1\|_{L^2} &= \|\mathcal{P}_{h,V}^* (1+|x|)^{-1} (1+|x|) E_h^1\|_{L^2}\\
&\leq\|\mathcal{P}_{h,V}^* (1+|x|)^{-1}\|_{L^2\rightarrow L^2} \|(1+|x|) E_h^1\|_{L^2}\\
&\leq C  \big{\|}(1+|x|)^{-1}\mathcal{P}_{h,V} \mathcal{P}_{h,V}^* (1+|x|)^{-1}\big{\|}_{L^2} ^{1/2} h^{(N-1)/2}\\
&= C  \big{\|}(1+|x|)^{-1} \big{(}(P_h- (1/2+i0))^{-1}) -(P_h- (1/2-i0))^{-1}) \big{)}  (1+|x|)^{-1}\big{\|}_{L^2}^{1/2} h^{(N+1)/2}\\
&\leq C h^{N/2},
\end{aligned}
\end{equation*}
by (\ref{resolvestim}). Here, we used the fact, whose proof can also be found in \cite[(2.26)]{Mel}, that $(P_h- (1/2+i0))^{-1}) -(P_h- (1/2-i0))^{-1})  = \frac{i}{2h} \mathcal{P}_{h,V} \mathcal{P}_{h,V}^*$.

All in all, we get that
$$\big{\|}\lim\limits_{|x|\rightarrow \infty} |x|^{-(d-1)/2} e^{-i |x|/h} E_h^1(|x|\hat{x})\big{\|}_{L^2} \leq C h^{N/2-1},$$
which proves the lemma.
\end{proof}

We have, by (\ref{passe}), that
\begin{equation*}
\begin{aligned}
&\lim\limits_{|x|\rightarrow \infty} |x|^{-(d-1)/2} e^{i |x|/h}\int_{-\infty}^{0} U_0(t) u^-_h e^{it/2h} \mathrm{d}t\\
&= \lim\limits_{|x|\rightarrow \infty} |x|^{-(d-1)/2} e^{i |x|/h}\int_{-\infty}^{t_-} U_0(t) u^0_h e^{i(t-t_-)/2h} \mathrm{d}t\\
&= e^{-it_-/2h} e^{i(d-1)\pi/4}  \frac{\pi^{1/2}}{(2\pi)^{d/2}} P_\Gamma\Big{(}{-\frac{\xi_0+\hat{x}}{\sqrt{h}}}\Big{)} e^{\frac{i}{h} x_0\cdot(\hat{x}+\xi_0)}  e^{-\frac{1}{2h} (\hat{x}+\xi_0) \cdot \overline{\Gamma}^{-1}(\hat{x}+\xi_0))}.
\end{aligned}
\end{equation*}

By (\ref{futur}) and (\ref{propagGauss2}), we have that
\begin{equation}
\begin{aligned}
&\Big{(}\lim\limits_{|x|\rightarrow \infty} |x|^{(d-1)/2} e^{-i|x|/h} \int_{0}^{+\infty} U_0(t) \tilde{u}^+_h e^{it/2h} \mathrm{d}t\Big{)}(\hat{x})\\
&= \frac{\pi^{1/2}}{(2\pi)^{d/2}}e^{i(1-d)\pi/4} e^{i\frac{\delta_+}{h}} \sum_{0\leq j\leq N} h^{j/2} (\pi_j)_{\Gamma_+} \Big{(}{\frac{\hat{x}-\xi_+}{\sqrt{h}}}\Big{)} e^{-\frac{i}{h} x_+\cdot (\hat{x}-\xi_+)}  e^{-\frac{1}{h} (\hat{x}-\xi_+) \cdot \Gamma_+^{-1}(\hat{x}-\xi_+)}.
\end{aligned}
\end{equation}

Finally, the last term composing $E_h^0$ can be dealt with using the following lemma.
\begin{lemme}
\begin{equation*}
\int_{0}^{t_+} U(t) u^-_h e^{it/2h} \mathrm{d}t = O(|x|^{-(d+1)/2}).
\end{equation*}
\end{lemme}
\begin{proof}
It suffices to show that, for any $t\in \R$, we have $ U(t) u^-_h = O(|x|^{-(d+1)/2})$. By Duhamel's principle, we have 
$$U(t) u_h^- = U_0(t) u_h^- - \frac{i}{h}\int_0^t U_0(t-s) V U(s) u_h^- \mathrm{d}s.$$
We have $U_0(t-s) V U(s) u_h^- = \mathcal{F}_h^{-1}\Big{[}  e^{i (t-s)|\xi|^2/2h} \mathcal{F}_h \Big{(}V U(s) u_h^-\Big{)}\Big{]}$. Since $V U(s) u_h^-$ is a compactly supported function, $e^{i (t-s)|\xi|^2/2h} \mathcal{F}_h \Big{(}V U(s) u_h^-\Big{)}$ is smooth and $L^2$, so that its inverse Fourier transform decays faster than any polynomial, and hence $- \frac{i}{h}\int_0^t U_0(t-s) V U(s) u_h^- \mathrm{d}s$ decays faster than any polynomial . A similar argument shows that $U_0(t) u_h^- $ decays faster than any polynomial, and the result follows.
\end{proof}

All in all, we have obtained that
$$
E_h(x)=|x|^{-(d-1)/2}\big{(} e^{-i |x|/h} E_h^{in}(-\hat{x}) +
e^{i |x|/h} E_h^{out}(\hat{x}) \big{)} + O(|x|^{-(d+1)/2}),
$$
with $$E_{in}(\hat{x})= e^{-it_-/2h} e^{i(d-1)\pi/4} \frac{\pi^{1/2}}{(2\pi)^{d/2}} P_\Gamma\Big{(}{\frac{\hat{x}-\xi_0}{\sqrt{h}}}\Big{)} e^{-\frac{i}{h} x_0\cdot(\hat{x}-\xi_0)}  e^{-\frac{1}{2h} (\hat{x}-\xi_0) \cdot \overline{\Gamma}^{-1}(\hat{x}-\xi_0))}$$
and $$E_h^{out}(\hat{x})= e^{i(1-d)\pi/4}\frac{\pi^{1/2}}{(2\pi)^{d/2}} e^{i\frac{\delta_+}{h}} \sum_{0\leq j\leq N} h^{j/2} (\pi_j)_{\Gamma_+} \Big{(}{\frac{\hat{x}-\xi_+}{\sqrt{h}}}\Big{)} e^{-\frac{i}{h} x_+\cdot (\hat{x}-\xi_+)}  e^{-\frac{1}{h} (\hat{x}-\xi_+) \cdot \Gamma_+^{-1}(\hat{x}-\xi_+))} + O(h^{N/2-1}).$$

We obtain from this that
\begin{equation*}
\begin{aligned}
&S_h\Big{(}P_\Gamma\Big{(}{\frac{\hat{x}-\xi_0}{\sqrt{h}}}\Big{)} e^{\frac{i}{h} x_0\cdot(\hat{x}-\xi_0)}  e^{-\frac{1}{2h} (\hat{x}-\xi_0) \cdot \overline{\Gamma}^{-1}(\hat{x}-\xi_0))}\Big{)} \\
&= e^{i t_-/2h} e^{i\frac{\delta_+}{h}} \sum_{0\leq j\leq N} h^{j/2} (\pi_j)_{\Gamma_+} \Big{(}{\frac{\hat{x}-\xi_+}{\sqrt{h}}}\Big{)} e^{\frac{i}{h} x_+\cdot(\hat{x}-\xi_+)}  e^{-\frac{1}{h} (\hat{x}-\xi_+) \cdot \Gamma_+^{-1}(\hat{x}-\xi_+))} + O(h^{(N-2+d)/2}).
\end{aligned}
\end{equation*}

Hence, by possibly replacing $N$ by $N+1$ if $d=2$ (which we can do, since $N$ was chosen arbitrary),  we obtain that
\begin{equation*}
\begin{aligned}
S_h\phi_{x_0,\xi_0,\overline{\Gamma}^{-1},P_\Gamma}
= e^{i\frac{\delta_1}{h}}\phi_{x_+,\xi_+,\Gamma_+,Q^N} + R_N,
\end{aligned}
\end{equation*}
where $Q^N = \sum_{0\leq j\leq N} h^{j/2} (\pi_j)_{\Gamma_+}$, $\delta_1 = \frac{t_-}{2}+ \delta_+ + x_0\cdot \xi_0- x_+\cdot\xi_+$, and  $\|R_N\|_{C^0}= O(h^{(N+1)/2})$.

Using the fact that $\Gamma\mapsto \overline{\Gamma}^{-1}$ and $P\mapsto P_\Gamma$ are one-to-one, this gives us the statement of the theorem.

\begin{remarque}
If the polynomial $P$ is constant equal to $1$, one could write an explicit formula for the constant $Q_1^0$ in terms of the classical dynamics by using (\ref{imageun}) and (\ref{enzero}).
\end{remarque}
\section{Resolution of identity}\label{sectionreso}

The aim of this section is to prove (\ref{resolution2}).
\begin{lemme}\label{resolution3}
Let $f\in L^2(\Sp^{d-1})$. We have
\begin{equation}\label{resolution}f(\omega) = c_h \int_{\Sp^{d-1}}\mathrm{d}\xi \int_{\xi^\perp} \mathrm{d}x  \tilde{\phi}_{x,\xi}(\omega) \int_{\Sp^{d-1}} \mathrm{d}\omega'  \overline{ \tilde{\phi}_{x,\xi}(\omega')} f(\omega'),
\end{equation}
where $c_h$ is a parameter depending on $h$, with $c_h\sim_{h\rightarrow 0}  2^{(d-1)/2}(2\pi h)^{-3(d-1)/2}$.
\end{lemme}
\begin{proof}
Let us write $g(\omega):= \int_{\Sp^{d-1}}\mathrm{d}\xi \int_{\xi^\perp} \mathrm{d}x \tilde{\phi}_{x,\xi}(\omega) \int_{\Sp^{d-1}} \mathrm{d}\omega'  \overline{ \tilde{\phi}_{x,\xi}(\omega')} f(\omega') $.

Let us write for all $t\geq 0$ $\tilde{\chi}_h(t) := e^{-\frac{t^2}{2h}} \chi \Big{(} \frac{t}{h^{1/3}}\Big{)}$, so that $\tilde{\chi}_h$ has support in $[0,\frac{3}{4}h^{1/3})$.

We have 
$$g(\omega)= \int_{\Sp^{d-1}}\mathrm{d}\xi  \int_{\Sp^{d-1}} \mathrm{d}\omega' \tilde{\chi}(|\omega-\xi|) \tilde{\chi}(|\omega'-\xi|)  f(\omega') \int_{\xi^\perp} \mathrm{d}x e^{\frac{i}{h} (-\omega+\omega')\cdot x}.$$

On $\{\omega \in \Sp^{d-1}; |\omega-\xi|\leq 3/4\}$, we may define a local chart by projecting on $\xi^\perp$. This way, we have $\omega=\omega(y)$ with $y\in \xi^\perp$, so that $|y|= |\sin (\widehat{\omega,\xi})|$, where $(\widehat{\omega,\xi})$ denotes the angle between the vectors $\omega$ and $\xi$. We also have $\Big{|} \det \Big{(}\frac{\mathrm{d}\omega}{dy}\Big{)}\Big{|} = |\cos (\widehat{\omega,\xi})|^{d-1}$.

Noting that if $x\in \xi^\perp$, we have $\omega(y)\cdot x= y\cdot x$, we get
\[\begin{aligned}
g(\omega)&=\int_{\Sp^{d-1}}\mathrm{d}\xi  \int_{\xi^\perp} |\cos (\widehat{\omega(y'),\xi})|^{d-1} \mathrm{d}y' \tilde{\chi}(|\omega-\xi|) \tilde{\chi}(|\omega(y')-\xi|)  f(\omega(y')) \\
&\int_{\xi^\perp} \mathrm{d}x e^{\frac{i}{h} (-y_{\omega}+y')\cdot x}\\
&= (2\pi h)^{d-1} f(\omega) \int_{\Sp^{d-1}}\mathrm{d}\xi  |\cos (\widehat{\omega,\xi})|^{d-1}  \tilde{\chi}^2(|\omega-\xi|).
\end{aligned}
\]
We write $c_h ^{-1}:= (2\pi h)^{d-1} \int_{\Sp^{d-1}}\mathrm{d}\xi  |\cos (\widehat{\omega,\xi})|^{d-1}  \tilde{\chi}^2(|\omega-\xi|)$, which is independent of $\omega$, since the integrand depends only on $|\omega-\xi|$.

Let us write $r(y)= |\omega(y)-\xi|$. We have $r(y) = y +o(y)$.

We have
\[\begin{aligned}
c_h^{-1} &=  (2\pi h)^{d-1} \int_{\xi^\perp} |\cos (\widehat{\omega(y),\xi})|^{d-1} e^{-\frac{|\omega(y)-\xi|^2}{h}} \Big{|}\chi\Big{(}\frac{|\omega(y)-\xi|}{h^{1/3}}\Big{)}\Big{|}^2 \mathrm{d}y\\
&=  (2\pi h)^{d-1} \int_{\xi^\perp} |\cos (\widehat{\omega(h^{1/2}z),\xi})|^{d-1} e^{-\frac{r(h^{1/2}z)^2}{h}} \Big{|}\chi\Big{(}\frac{r(h^{1/2}z)}{h^{1/3}}\Big{)}\Big{|}^2 h^{(d-1)/2}\mathrm{d}z\\
&= (2\pi)^{d-1} h^{3(d-1)/2} \int_{\xi^\perp} e^{-z^2} \mathrm{d}z + o(h^{(d-1)/2})\\
&= 2^{-(d-1)/2}(2\pi h)^{3(d-1)/2} + o(h^{(d-1)/2}).
\end{aligned}\]
\end{proof}

As a corollary of the resolution of identity formula, let us state the following result, which can be proved along the same line as in \cite[1.2.3]{combescure2012coherent}
\begin{corolaire}
Let $A$ be a trace-class operator acting on $L^2(\Sp^{d-1})$. We then have
$$\mathrm{Tr} A = c_h \int_{\Sp^{d-1}} \mathrm{d}\omega \int_{\omega^\perp} \mathrm{d}\xi \langle\tilde{\phi}_{\omega,\xi}, A \tilde{\phi}_{\omega,\xi}\rangle_{L^2(\Sp^{d-1})}.$$
\end{corolaire}

\bibliographystyle{alpha}
\bibliography{references}
\end{document}